\documentclass[10pt]{amsart}
\usepackage{hyperref}
\hypersetup{nesting=true,debug=true,naturalnames=true}
\usepackage{graphicx,amssymb,upref}
\usepackage{amsmath,amssymb,amsthm,multicol,dsfont,mathtools,url,enumitem,setspace,rotating,graphicx,wasysym,stmaryrd}
 \usepackage[mathscr]{eucal}
\usepackage{amscd} \usepackage{setspace} \usepackage{epsfig}
 \usepackage{dsfont}\usepackage{tensor}
\usepackage{mathtools,placeins}
\usepackage{multicol,hyperref,enumitem,pinlabel}
\usepackage[usenames,dvipsnames]{color}
\usepackage{float}
\usepackage{tikz}
\usetikzlibrary{matrix}

\newcommand\br{\text{bth}}
\newcommand\0{\textbf{0}}

\newcommand\x{\textbf{x}}

\newcommand\z{\textbf{z}}

\newcommand\Z{\mathds{Z}}

\newcommand\lla{\left\langle}
\newcommand\rra{\right\rangle}

\newcommand\lk{\text{lk}}

\newcommand\bbm{\begin{bmatrix}}
\newcommand\ebm{\end{bmatrix}}
\newcommand\bmat{\begin{matrix}}
\newcommand\emat{\end{matrix}}
\newcommand\red[1]{\color{red}#1\color{black}}

\newcommand\Cyan[1]{\color{Cyan}#1\color{black}}

\newcommand\PosCr{\raisebox{-2pt}{\includegraphics[height=11pt]{PositiveCrossing.pdf}}}
\newcommand\NegCr{\raisebox{-2pt}{\includegraphics[height=11pt]{NegativeCrossing.pdf}}}

\newcommand\CBandPos{\raisebox{-3pt}{\includegraphics[height=11pt]{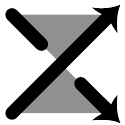}}}
\newcommand\CBandNeg{\raisebox{-3pt}{\includegraphics[height=11pt]{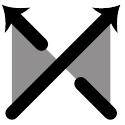}}}

\newcommand\SeifAlgPos{\raisebox{-3pt}{\includegraphics[height=11pt]{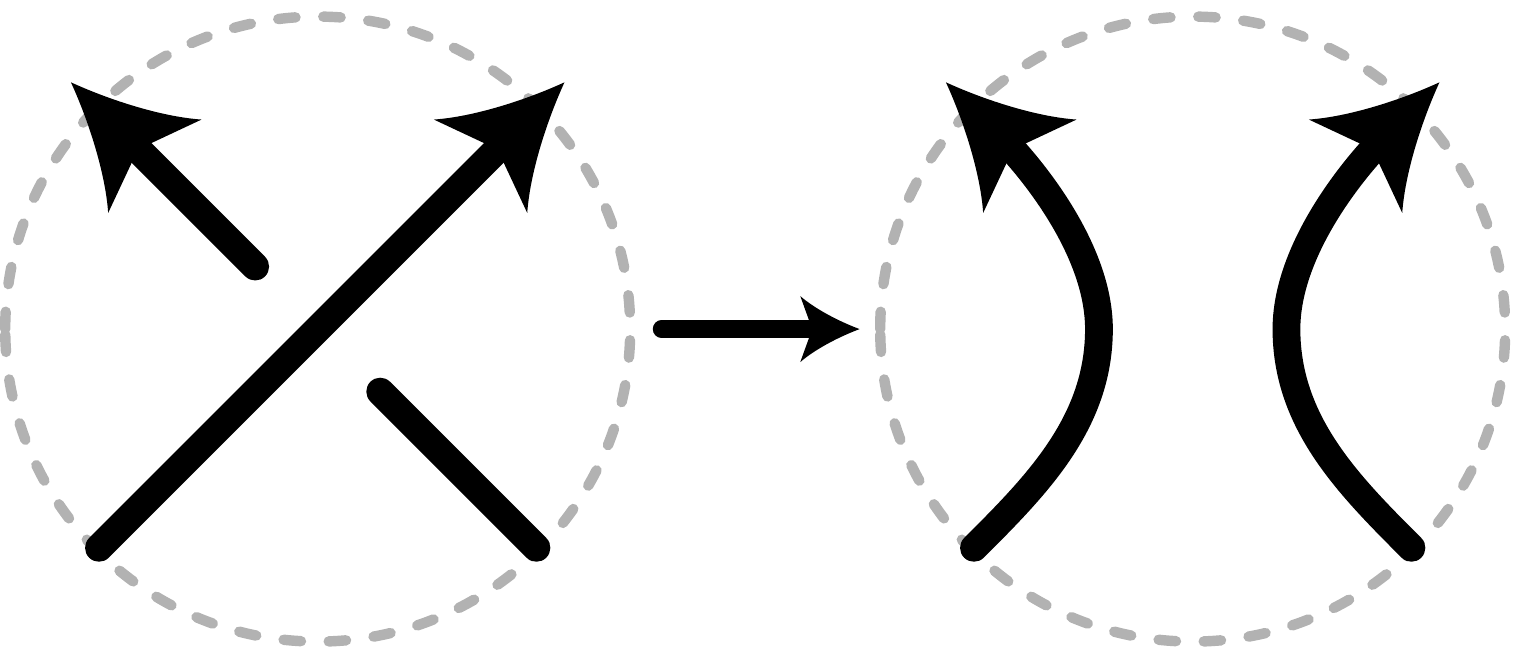}}}
\newcommand\SeifAlgNeg{\raisebox{-3pt}{\includegraphics[height=11pt]{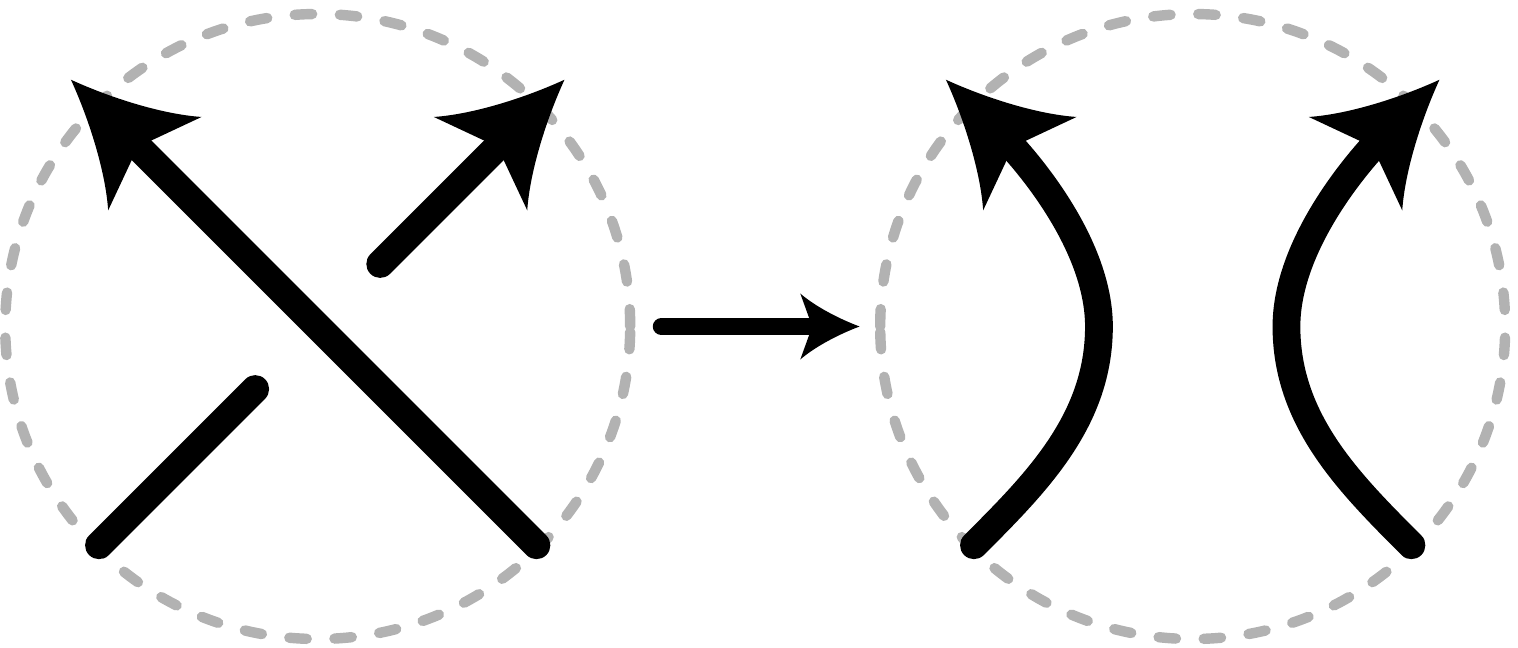}}}

\theoremstyle{plain}
\newtheorem{theorem}{Theorem}
\newtheorem{lemma}[theorem]{Lemma}

\newtheorem{prop}[theorem]{Proposition}
\newtheorem{cor}[theorem]{Corollary}

\theoremstyle{definition}

\newtheorem{definition}[theorem]{Definition}
\newtheorem{question}[theorem]{Question}

\newtheorem{problem}[theorem]{Problem}

\theoremstyle{remark}

\author{Thomas Kindred}

\address{Department of Mathematics, Wake Forest University \\
Winston-Salem, North Carolina 27109, USA} 

\email{kindret@wfu.edu}
\urladdr{www.thomaskindred.com}

\title{A simple proof of the Crowell--Murasugi theorem}
\date{\today}

\begin{document}

\maketitle

\begin{abstract}
We give an elementary, self-contained proof of the theorem, proven independently in 1958-9 by Crowell and Murasugi, that the genus of an alternating knot equals half the breadth of its Alexander polynomial, and that applying Seifert's algorithm to any alternating knot diagram gives a surface of minimal genus. 

\end{abstract}

Every oriented knot $K\subset S^3$ bounds an oriented surface $F$ called a {\it Seifert surface}.  Such $F$ is homeomorphic to a once-punctured surface of some genus $g(F)$. The {\bf knot genus} $g(K)$ is the minimum genus among all Seifert surfaces for $K$.  

An ordered basis $(a_1,\hdots,a_n)$ for $H_1(F)$ determines an $n\times n$ {\it Seifert matrix} $V=(v_{ij})$, $v_{ij}=\lk(a_i,a_j^+)$, where $\lk$ denotes linking number and $a_j^+$ is the pushoff of (an oriented multicurve representing) $a_j$ in the positive normal direction determined by the orientations on $F$ and $S^3$.
The polynomial $\det(V-tV^T)$, denoted $\Delta_K(t)$, is called the {\it Alexander polynomial} of $K$. Up to sign and degree shift, it is independent of Seifert surface and basis.  
Writing $\Delta_K(t)=a_rt^r+a_{r+1}t^{r+1}+\cdots+a_{s-1}t^{s-1}+a_st^s$ with $a_r,a_s\neq 0$, the value $s-r$ is called the {\bf breadth} of $\Delta_K(t)$; we denote this by $\br(K)$.

Given any oriented diagram $D\subset S^2$ of a knot $K\subset S^3$, {\bf Seifert's algorithm} produces a Seifert surface for $K$ as follows.  First, ``smooth" each crossing of $D$ in the way that respects orientation: \SeifAlgPos, \SeifAlgNeg. This gives a disjoint union of oriented circles on $S^2$ called the {\it Seifert state} of $D$; each circle is called a {\it Seifert circle}. Second, cap all the Seifert circles with disjoint, oriented disks, all on the same side of $S^2$.  Third, attach an oriented half-twisted band at each crossing, so that the resulting surface $F$ is oriented with $\partial F=K$, respecting orientation. 
Here is an example:
\[\includegraphics[width=\textwidth]{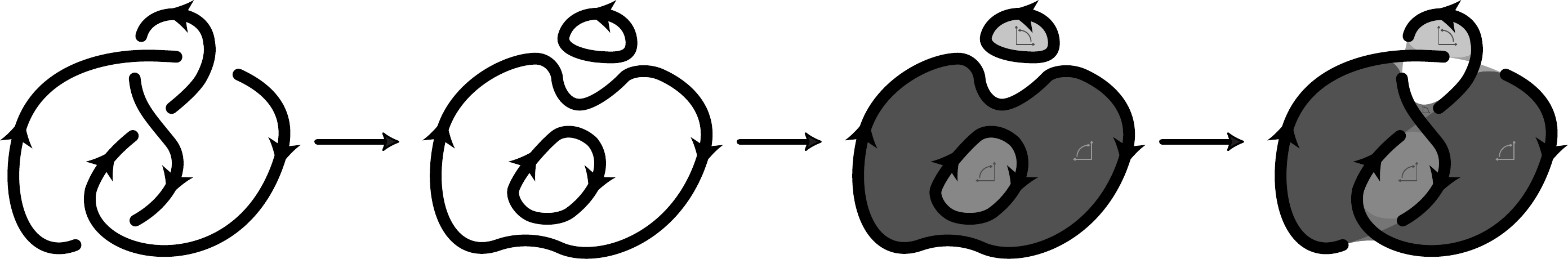}\]
The purpose of this note is to give a short, elementary, self-contained proof of the following theorem, first proven independently in 1958-9 by Crowell and Murasugi:

\begin{theorem}[\cite{crowell,mur58}]\label{T:CroMur}
If $F$ is a surface constructed via Seifert's algorithm from an alternating diagram $D$ of a knot $K$, then 
\[g(F)=g(K)=\frac{1}{2}\br(K).\]
\end{theorem}

To prove Theorem \ref{T:CroMur}, we will show that a Seifert matrix $V$ for $F$ is invertible.  The next two results show that this indeed will suffice:

\begin{prop}\label{P:BreadthSuffices}
Let $F$ be a Seifert surface for a knot $K$.  If $\br(K)=2g(F)$, then $g(K)=g(F)=\frac{1}{2}\br(K)$.
\end{prop}

\begin{proof}
Let $F'$ be an arbitrary Seifert surface for $K$. One may compute $\Delta_K(t)$ from any Seifert matrix for $F'$, so $\br(K)\leq 2g(F')$. Hence $g(F)\leq g(F')$. 
\end{proof}

\begin{prop}\cite{mur96}\label{P:BreadthInv}
Let $V$ be a real $2g\times 2g$ matrix, and let $f(t)=\det(V-tV^T)$.  If $V$ is invertible, then the breadth of $f(t)$ equals $2g$.
\end{prop}

\begin{proof}
Denoting the transpose of $V^{-1}$ by $V^{-T}$, \begin{equation*}
f(t)=\det(V^{-T})\det(VV^{-T}-tI)
\end{equation*}
is a nonzero scalar multiple of the characteristic polynomial of the invertible matrix $VV^{-T}$, hence has breadth $2g$.
\footnote{The converse is also true.  Indeed, if $V$ is singular, then choose an invertible matrix $P$ whose first column is in the nullspace of $V$.  Then $\det(P^TVP-t(P^TVP)^T)=\det^2(P)\cdot f(t)$ has the same breadth as $f(t)$. Further, the first column of $P^TVP$ is $\0$, so only constants appear in the first row of $P^TVP-t(P^TVP)^T$. Hence, the breadth is less than $2g$.}
\end{proof}

Next, suppose that $D\subset S^2$ is a connected oriented alternating {\it link} diagram, and that applying Seifert's algorithm to $D$ yields a checkerboard surface $F$.\footnote{That is, each Seifert circle bounds a disk in $S^2$ disjoint from the other Seifert circles.} 
Then, since $D$ is alternating, all of the crossing bands in $F$ are identical: either they all positive, \CBandPos, or they are all negative, \CBandNeg. 

\begin{lemma}\label{L:CBDef}
With the setup above, if the crossing bands in $F$ are positive, then any nonzero $\x\in\Z^{\beta_1(F)\times \beta_1(F)}$ satisfies $\x^TV\x>0$; if the crossing bands in $F$ are negative, then any such $\x$ satisfies $\x^TV\x<0$. Hence, in either case, $V$ is invertible.
\end{lemma}

Here is a self-contained proof. A shorter argument, using \cite{greene}, follows.
\begin{proof} 
Assume without loss of generality that the crossing bands in $F$ are positive.  Among all oriented multicurves in $F$ that represent $\x$, choose one, $\alpha$, that intersects the crossing bands in $F$ in the smallest possible number of components.  Then, for each crossing band $X$ in $F$, $\alpha\cap X$ will consist of a (possibly empty) collection of coherently oriented arcs. Therefore:
\begin{equation}\label{E:DefCurve}
\x^TV\x=\lk(\alpha,\alpha^+)=\sum_{\text{crossing bands }X}\frac{|\alpha\cap X|^2}{2}\geq 0.
\end{equation}
Moreover, the inequality in (\ref{E:DefCurve}) is strict, or else $\alpha$ would be disjoint from all crossing bands, hence nullhomologous. It follows that $V$ is nonsingular, or else we would have $V\z=\0$ for some nonzero vector $\z$, giving $\z^TV\z=0$.
\end{proof}

Alternatively, denote the Gordon-Litherland pairing on $F$ by $\lla\cdot,\cdot\rra$ \cite{gordlith}. Since $D$ is alternating, this pairing is definite \cite{mur87ii,greene}.  Thus:
\begin{equation*}
\x^TV\x=\lk(\alpha,\alpha^+)=\frac{1}{2}\lk(\alpha,\alpha_+\cup\alpha_-)=\lla \x,\x\rra\neq 0.
\end{equation*}

To complete the proof of Theorem \ref{T:CroMur}, we need one more definition and lemma.  Murasugi sum, also called generalized plumbing, is a way of gluing together two spanning surfaces along a disk so as to produce another spanning surface.  We will prove that if Seifert surfaces $F_1$ and $F_2$ have invertible Seifert matrices, then any Murasugi sum of $F_1$ and $F_2$ also has invertible Seifert matrix (and conversely).

\begin{definition}\label{D:MurSum}
For $i=1,2$, let $F_i$ be a Seifert surface in a 3-sphere $S^3_i$, and choose a compact 3-ball $B_i\subset S^3_i$ that contains $F_i$ such that:
\begin{itemize}
\item $F_i\cap \partial B_i$ is a disk $U_i$ whose boundary consists alternately of arcs in $\partial F_i$ and arcs in $\text{int}(F_i)$;
\item $|\partial U_1\cap \partial F_1|=|\partial U_2\cap \partial F_2|$, where bars count components; and
\item the positive normal along $U_1$ (using the orientations on $S^3_1$ and $F_1$) points {\it into} $B_1$, and the positive normal along $U_2$ points {\it out of} $B_2$.
\end{itemize}
Choose an orientation-reversing homeomorphism $h:\partial B_1\to\partial B_2$ such that 
\begin{itemize}
\item $h(U_1)=U_2$ and
\item $h(\partial U_1\cap \partial F_1)=\text{cl}(\partial U_2\cap \text{int}(F_2))$.\footnote{It follows that $h(\text{cl}(\partial U_1\cap \text{int}(F_1)))=\partial U_2\cap \partial F_2$.}
\end{itemize}
Then $F=F_1\cup_hF_2$ is a Seifert surface in the 3-sphere $B_1\cup_hB_2$.  It is a {\bf Murasugi sum} or {\it generalized plumbing} of $F_1$ and $F_2$, denoted $F=F_1*F_2$.
\end{definition}

Note that there are generally many ways to form a Murasugi sum between two given surfaces. As an aside, we mention that the Murasugi sum construction extends easily to unoriented surfaces, and that both the oriented and unoriented notions of Murasugi sum are natural operations in many respects \cite{gab1,gab2,ozawa11,ozbpop,essence}. Here is one such respect:

\begin{lemma}\label{L:MurV}
Suppose $F=F_1*F_2$ is a Murasugi sum of Seifert surfaces, and denote the respective Seifert matrices by $V$, $V_1$, and $V_2$.  Then $V$ is invertible if and only if both $V_1$ and $V_2$ are invertible.
\end{lemma}

\begin{proof}
Denote $V=(v_{ij})$. We may assume that $V$ is taken with respect to a basis $(a_1,\hdots,a_r,b_1,\hdots,b_s)$ for $H_1(F)$, where $(a_1,\hdots,a_r)$ is a basis for $H_1(F_1)$ and $(b_1,\hdots,b_s)$ is a basis for $H_1(F_2)$.  Then $V$ is a block matrix of the form $V=$\scalebox{.7}{$\bbm V_1&A\\B&V_2 \ebm$}.  In fact, we claim that $B=0$, i.e.
\begin{equation}\label{E:SeifBlock}
V=\bbm V_1&A\\0&V_2 \ebm.
\end{equation}
To see this, let $\alpha_j\subset F_1$ represent $a_j$ and let $\beta_i\subset F_2$ represent $b_i$ for arbitrary $1\leq j\leq r$, $1\leq i\leq s$.  Then $v_{ij}=\lk (\beta_i,\alpha_j^+)=0$ because, using the notation and setup from Definition \ref{D:MurSum}, $\alpha_j^+\subset\text{int}(h(B_1))$ and $\beta_i\subset B_2$.  From (\ref{E:SeifBlock}), we have $\det(V)=\det(V_1)\det(V_2)$,\footnote{This follows from the formula $\displaystyle\det(V)=\!\!\!\!\sum_{\sigma\in S_{r+s}}\!\!\!\text{sign}(\sigma)\prod_{i=1}^{r+s}v_{i\sigma(i)}$ and the pigeonhole principle.} so the result follows.
\end{proof}

Now we can prove Theorem \ref{T:CroMur}:

\begin{proof}[Proof of Theorem \ref{T:CroMur}]
Let $F$ be a surface constructed via Seifert's algorithm from an alternating diagram $D$ of a knot $K$.  Then $F$ is a Murasugi sum of checkerboard Seifert surfaces from connected alternating link diagrams.  

Lemma \ref{L:CBDef} implies that all of these checkerboard surfaces have invertible Seifert matrices, so Lemma \ref{L:MurV} implies that $F$ also has an invertible Seifert matrix $V$.  Since $K$ is a knot, the size of $V$ is $\beta_1(F)=2g(F)$.  Thus, by Propositions \ref{P:BreadthSuffices} and \ref{P:BreadthInv},
\[\pushQED{\qed}
g(F)=g(K)=\frac{1}{2}\br(K).\qedhere\]
\end{proof}

The proof above shows more generally:

\begin{theorem}\label{T:gen}
Let $F$ be a Seifert surface for a knot $K$.  If $F$ is a Murasugi sum of checkerboard surfaces from connected alternating link diagrams, then $g(K)=g(F)=\frac{1}{2}\br(K)$.  
\end{theorem}

In particular, a knot diagram is called {\bf homogeneous} if it is a {\it $*$-product}, i.e. {\it diagrammatic Murasugi sum}, of special alternating link diagrams.  By definition, Theorem \ref{T:gen} applies to all such diagrams (c.f. \cite{c89} Corollary 4.1):

\begin{cor}
Let $F$ be a surface constructed via Seifert's algorithm from a {\it homogeneous} diagram of a knot $K$.  Then $g(F)=g(K)=\frac{1}{2}\br(K)$.
\end{cor}

We note another consequence of Lemma \ref{L:MurV}, in combination with:

\begin{theorem}[Harer's conjecture \cite{harer}; Corollary 3 of \cite{girgoo}]
Any fiber surface in $S^3$ can be constructed by plumbing and deplumbing Hopf bands.  
\end{theorem}

\begin{cor}
If $F$ is a fiber surface spanning a knot $K\subset S^3$, then  
\[g(F)=g(K)=\frac{1}{2}\br(K).\] 
\end{cor}

\begin{figure}
\begin{center}\labellist
\tiny\hair 4pt
\pinlabel {de-plumb} [b] at 339 170
\pinlabel {1 band} [t] at 337 175
\pinlabel {isotope} [b] at 727 170
\pinlabel {de-plumb} [b] at 542 -195
\pinlabel {3 bands} [t] at 542 -190
\endlabellist
\includegraphics[height=.25\textwidth]{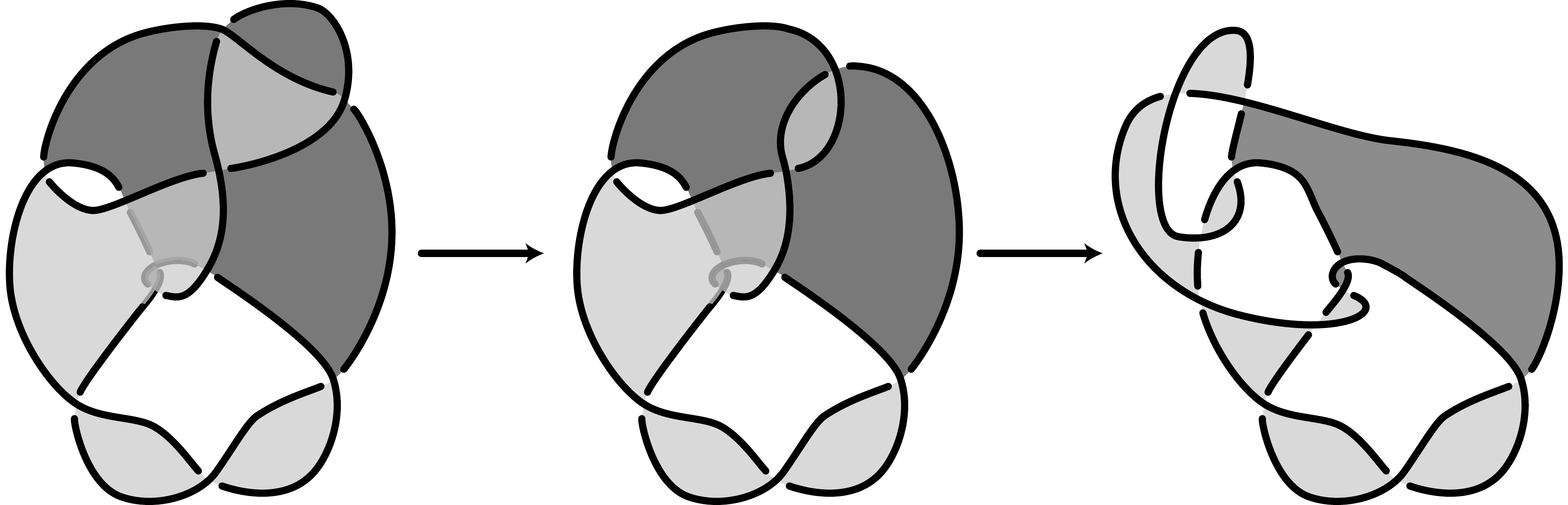}
\includegraphics[height=.25\textwidth]{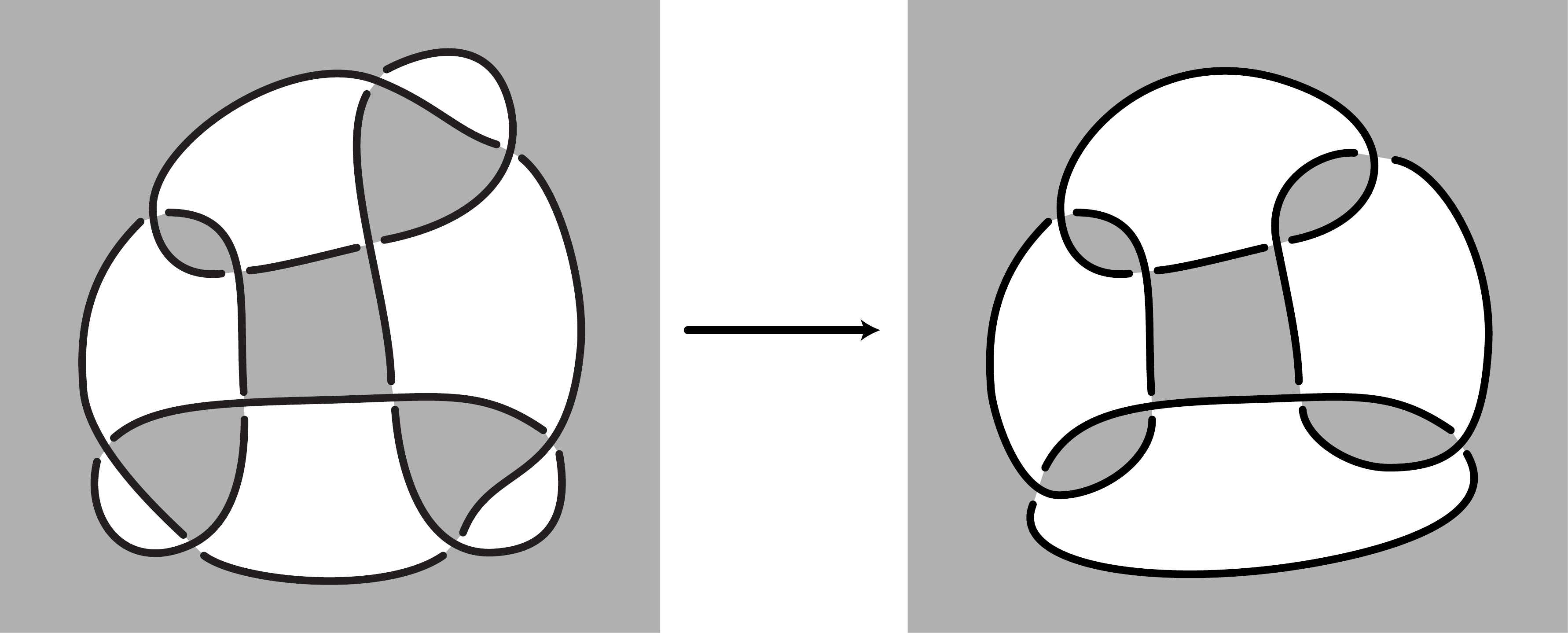}
\caption{Deplumbing Hopf bands from minimal genus Seifert surfaces for the knots $11n67$ and $11n73$}
\label{Fi:deplumb}
\end{center}
\end{figure}
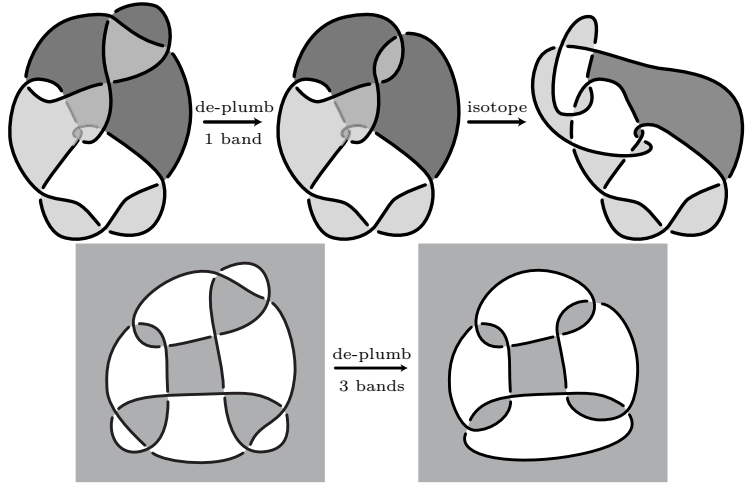

We close by considering knots $K$ with $g(K)>\frac{1}{2}\br(K)$.  The simplest such knots have 11 crossings.  There are seven of them \cite{knotinfo}: the Conway knot $11n34$ has genus three, as do $11n45$, $11n73$, and $11n152$, while the Kinoshita-Terasaka knot $11n42$ has genus two, as do $11n67$ and $11n97$.  Lemma \ref{L:MurV} implies that if one takes a minimal genus Seifert surface for any one of these knots and de-plumbs (i.e. decomposes it as a nontrivial Murasugi sum),\footnote{Beware: surfaces may admit distinct de-plumbings \cite{essence}.  Still, Lemma \ref{L:MurV} implies that this sentence is true for {\it any} de-plumbing of such a surface.} then at least one of the resulting factors will have singular Seifert matrix. Also, by Theorem 1 of \cite{gab2}, all of these surfaces will have minimal genus. This raises the following natural problem:

\begin{problem}\label{P:Sing}
Characterize or tabulate those Seifert surfaces $F$ which (i) have minimal genus, (ii) do not deplumb,\footnote{That is, any decomposition of $F$ as a Murasugi sum $F=F_1*F_2$ has $F_1$ or $F_2$ as a disk.} and (iii) have singular Seifert matrices.
\end{problem}

Interestingly, for each of the four aforementioned 11-crossing knots of genus three, de-plumbing a minimal genus Seifert surface gives three Hopf bands and the planar pretzel surface $P_{2,2,-2,-2}$, which has Seifert matrix 
\[\bbm 2&-1&0\\ -1&0&1\\0&1&-2 \ebm,\]
and doing this for any of the three aforementioned 11-crossing knots of genus two gives one Hopf band and a surface of genus one that has Seifert matrix
\[\bbm 0&1&0\\ 1&0&-2\\0&-1&0 \ebm.\]
See Figure \ref{Fi:deplumb}. Another simple example of the type of surface referenced in Problem \ref{P:Sing} is the planar pretzel surface $P_{4,4,-2}$, which has Seifert matrix
\[\bbm 4&-2\\-2&1\ebm.\]
In particular, each of these simplest examples spans a link of multiple components. 

\begin{question}
Does there exist a {\it knot} $K$ with $g(K)>\frac{1}{2}\br(K)$ with a minimal genus Seifert surface $F$ that does not deplumb?
\end{question}


\begin{thebibliography}{99}
\bibitem[1]{knotinfo} \url{https://www.indiana.edu/~knotinfo/}
%
\bibitem[Cr89]{c89} P.R. Cromwell, {\it Homogeneous links}, J. London Math. Soc. (2) 39 (1989), no. 3, 535-552.
%
\bibitem[Cr59]{crowell} R. Crowell, {\it Genus of alternating link types}, Ann. of Math. (2) 69 (1959), 258-275.
%
\bibitem[Ga83]{gab1} D. Gabai, {\it The Murasugi sum is a natural geometric operation},  Low-dimensional topology (San Francisco, Calif., 1981), 131-143, Contemp. Math., 20, Amer. Math. Soc., Providence, RI, 1983.
%
\bibitem[Ga85]{gab2} D. Gabai, {\it The Murasugi sum is a natural geometric operation II},  Combinatorial methods in topology and algebraic geometry (Rochester, N.Y., 1982), 93-100, Contemp. Math., 44, Amer. Math. Soc., Providence, RI, 1985. 
%
\bibitem[Ga86]{gab86alt} D. Gabai, {\it Genera of the alternating links}, Duke Math J. Vol 53 (1986), no. 3, 677-681.
%
\bibitem[GG06]{girgoo} E. Giroux, N. Goodman, {\it On the stable equivalence of open books in three-manifolds}, Geom. Topol. 10 (2006), 97-114.
%
\bibitem[GL78]{gordlith} C.M. Gordon, R.A. Litherland, {\it On the signature of a link}, Invent. Math. 47 (1978), no. 1, 53-69.
%
\bibitem[Gr17]{greene} J. Greene, {\it Alternating links and definite surfaces}, Duke Math. J. 166 (2017), no. 11, 2133-2151. arXiv:1511.06329v1.
%
\bibitem[Ha82]{harer} J. Harer, {\it How to construct all fibered knots and links}, Topology 21 (1982), no. 3, 263-280.
%
\bibitem[Ki22]{essence} T. Kindred, {\it Essence of a spanning surface}, preprint.
%
\bibitem[Mu58]{mur58} K. Murasugi, {\it On the genus of the alternating knot. I, II}, J. Math. Soc. Japan 10 (1958), 94-105, 235-248.
%
\bibitem[Mu87ii]{mur87ii} K. Murasugi, {\it Jones polynomials and classical conjectures in knot theory II}, Math. Proc. Cambridge Philos. Soc. 102 (1987), no. 2, 317-318.
%
\bibitem[Mu96]{mur96} K. Murasugi, {\it Knot theory and its applications}, translated from the 1993 Japanese original by Bohdan Kurpita. Birkh\"auser Boston, Inc., Boston, MA, 1996. viii+341 pp.
%
\bibitem[Oz11]{ozawa11} M. Ozawa, {\it Essential state surfaces for knots and links}, J. Aust. Math. Soc. 91 (2011), no. 3, 391-404. 
%
\bibitem[OP16]{ozbpop} B. Ozbagci, P. Popescu-Pampu, {\it Generalized plumbings and Murasugi sums}, Arnold Math. J. 2 (2016), no. 1, 69-119.

\end{thebibliography}
\end{document}